\documentclass[preprint,12pt]{elsarticle}

\usepackage[english]{babel}
\usepackage{amsmath}
\usepackage{amsfonts}
\usepackage{amssymb}
\usepackage{amsthm}
\usepackage{enumitem}
\usepackage[hidelinks]{hyperref}

\journal{Journal of Algebra}

\theoremstyle{plain}
\newtheorem{theorem}{Theorem}[section]
\newtheorem{proposition}[theorem]{Proposition}
\newtheorem{lemma}[theorem]{Lemma}
\newtheorem{corollary}[theorem]{Corollary}
\newtheorem{fact}[theorem]{Fact}
\newtheorem*{maintheorem}{Theorem}
\newtheorem{question}{Question}

\theoremstyle{definition}
\newtheorem{definition}[theorem]{Definition}

\theoremstyle{remark}
\newtheorem{remark}[theorem]{Remark}

\def\term#1{\emph{#1}}
\def\wt#1{\widetilde{#1}}
\def\wx{\wt{X}}
\def\N{\mathbb{N}}
\def\Z{\mathbb{Z}}
\def\R{\mathbb{R}}
\def\Rp{\R_+}
\def\ep{\varepsilon}
\def\hd{\hat{d}}
\def\vd{\tilde{d}}
\def\cl#1{\overline{#1}}
\def\ph{\varphi}
\def\al{\alpha}
\def\be{\beta}
\def\ga{\gamma}
\def\om{\omega}

\def\env#1{\langle #1 \rangle}
\def\one{\mathbf{1}}
\def\FI{F([0,1],0)}
\newcommand{\supp}{\operatorname{supp}}
\newcommand{\ab}{\operatorname{ab}}

\begin{document}

\begin{frontmatter}

\title{On identities in connected topological groups}

\author[zy]{Ilya N. Zyabrev}
\ead{inzyabrev@gmail.com}
\affiliation[zy]{city={Moscow},
            country={Russia}}

\author[msu,smbu]{Evgenii A. Reznichenko\corref{cor1}}
\ead{erezn@inbox.ru}
\cortext[cor1]{Corresponding author.}
\affiliation[msu]{organization={Department of General Topology and Geometry, Faculty of Mechanics and Mathematics, Lomonosov Moscow State University},
            addressline={Leninskie Gory 1},
            city={Moscow},
            postcode={119991},
            country={Russia}}
\affiliation[smbu]{organization={Shenzhen MSU-BIT University},
            addressline={1 International University Park Road, Dayun New Town, Longgang District},
            city={Shenzhen},
            postcode={518172},
            state={Guangdong},
            country={China}}

\begin{abstract}
We construct a connected (indeed, path-connected) Hausdorff topological group $G$ such that, for an odd $n\ge665$, the identity $x^n=e$ holds in some neighbourhood of the identity but does not hold in the whole group. This answers in the negative a question of Mycielski and Aronszajn (Question~2.48 of the Kourovka Notebook).
\end{abstract}


\end{frontmatter}

\section*{Introduction}

In 1957 Mycielski \cite{Mycielski1957} proved that if an identity holds in some neighbourhood of the identity of a connected locally compact group or of a connected abelian topological group, then it holds in the whole group. He also asked the following question: let $G$ be a connected topological group such that the identity $x^3=e$ holds in some neighbourhood of the identity of $G$; does the identity $x^3=e$ hold in the whole group $G$? The same question was asked for the identity $gx^2=x^2g$, where $g$ is a fixed element of the group.

The general question (due to N.~Aronszajn) was included by V.\,P.~Platonov in the 2nd issue of the Kourovka Notebook (1967) as Question~2.48 \cite{kour_tetrad2014,kourovka_arxiv}:
\emph{let $G$ be a connected topological group, let $f$ be an identity, and let $f|_U=e$ for some neighbourhood $U$ of the identity of $G$. Is it true that then $f|_G=e$?}

In this paper we answer this question in the negative.

\begin{maintheorem}
Let $n\ge665$ be an odd number. There exist a path-connected Hausdorff topological group $G$ and an open neighbourhood $V$ of the identity of $G$ such that $x^n=e$ for all $x\in V$, but $x^n\neq e$ for some $x\in G$.
\end{maintheorem}

The group $G$ is constructed as the quotient of the free topological group $\FI$ by the closure of the subgroup generated by the $n$th powers of elements of small Graev norm. The proof consists of two independent parts: a topological one (Proposition~\ref{p:A}: the closure of such a subgroup is not much larger than the subgroup itself) and an algebraic one (Proposition~\ref{p:B}: a certain $n$th power does not belong to this subgroup). The second part reduces to the existence, in the free Burnside group $B(m,n)$, of a system of periods with two properties (Definition~\ref{d:sp}): every element is conjugate to a power of a period, and the $n$th powers of the periods form a basis of the centre of the corresponding central extension. We give two proofs of this part. The first one relies on results of A.\,Yu.~Ol'shanskii \cite{Olshanskii1982,Olshanskii1989book} and of I.\,S.~Ashmanov and A.\,Yu.~Ol'shanskii \cite{AshmanovOlshanskii1985} and works for odd $n>10^{10}$. The second one relies on results of S.\,I.~Adian \cite{Adian1975} and of S.\,I.~Adian and V.\,S.~Atabekyan \cite{AdianAtabekyan2018} and works for odd $n\ge665$.

The present paper is a corrected and substantially simplified version of the preprint \cite{ZyabrevReznichenko2024}, which is referred to in the comment to Question~2.48 in \cite{kourovka_arxiv}. In the preprint and in that comment the theorem is stated for odd $n\ge10^{10}$. The result was announced in 1985 at the research seminar on general topology of Moscow State University \cite{Seminar1985}.

\section{Notation and preliminaries}
\label{sec:prelim}

We denote by $\N=\{1,2,\dots\}$ the set of natural numbers, by $\om_1$ the first uncountable ordinal, and $\Rp=[0,+\infty)$. The identity of a group is denoted by $e$, and $\env M$ is the subgroup generated by a set $M$. The closure of a set $M$ in a topological space is denoted by $\cl M$. All topological groups under consideration are Hausdorff.

\subsection{Norms and invariant metrics}

Let $K$ be a group. A function $\nu\colon K\to\Rp$ is called an \term{invariant pseudonorm} if $\nu(e)=0$, $\nu(xy)\le \nu(x)+\nu(y)$, $\nu(x^{-1})=\nu(x)$, and $\nu(yxy^{-1})=\nu(x)$ for all $x,y\in K$; if, in addition, $\nu(x)=0$ only for $x=e$, then $\nu$ is called an \term{invariant norm}. A pseudometric $\rho$ on $K$ is called \term{invariant} if $\rho(zx,zy)=\rho(xz,yz)=\rho(x,y)$ for all $x,y,z\in K$. The formulas $\nu(x)=\rho(x,e)$ and $\rho(x,y)=\nu(xy^{-1})$ establish a one-to-one correspondence between invariant pseudometrics and invariant pseudonorms (metrics and norms, respectively). A set $M\subset K$ is called \term{invariant} if $gMg^{-1}=M$ for all $g\in K$.

\subsection{Free topological groups}

Let $X$ be a Tychonoff space and $e\in X$. The \term{free (Graev) topological group} $F(X,e)$ is a topological group that contains $X$ as a subspace, has $e$ as its identity, and is such that every continuous map $f\colon X\to K$ to a topological group $K$ with $f(e)=e$ extends uniquely to a continuous homomorphism $\hat f\colon F(X,e)\to K$ \cite{Graev1948}. Algebraically, $F(X,e)$ is the free group with basis $X\setminus\{e\}$. For $g\in F(X,e)$, we denote by $\supp g$ the set of elements of $X\setminus\{e\}$ occurring in the reduced word representing $g$. If $e\in T\subset X$, then $F(T,e)$ denotes the subgroup $\env T\subset F(X,e)$; it is the free group with basis $T\setminus\{e\}$.

Let $\wx=X\cup X^{-1}\subset F(X,e)$; then $X\cap X^{-1}=\{e\}$. For $k\in\N$, let
\[
j_k\colon \wx^k\to F(X,e),\qquad (x_1,\dots,x_k)\mapsto x_1x_2\cdots x_k,
\]
and $F_k(X,e)=j_k(\wx^k)$. The map $j_k$ is continuous, and $F(X,e)=\bigcup_{k\in\N}F_k(X,e)$.

\begin{fact}[\cite{Graev1948,MackMorrisOrdman1973}]\label{f:klim}
Let $X$ be compact. Then the sets $F_k(X,e)$ are compact, and a set $M\subset F(X,e)$ is closed in $F(X,e)$ if and only if $M\cap F_k(X,e)$ is closed for all $k\in\N$.
\end{fact}
In \cite[Theorem~7.4.1 and Corollary~7.4.2]{ArhangelskiiTkachenko2008} this is proved for the Markov free topological group $F(X)$ and its subsets $B_k(X)$ of elements of reduced length $\le k$. The assertion for $F(X,e)$ follows. Indeed, the homomorphism $\pi\colon F(X)\to F(X,e)$ extending the identity map of $X$ is a quotient map: the group $F(X)/\ker\pi$ with the quotient topology has the universal property of $F(X,e)$. Moreover, $\pi(B_k(X))=F_k(X,e)$. If the sets $M\cap F_k(X,e)$ are closed for all $k$, then so are the sets $\pi^{-1}(M)\cap B_k(X)=(\pi|_{B_k(X)})^{-1}(M\cap F_k(X,e))$; hence $\pi^{-1}(M)$ is closed, and therefore $M$ is closed.

If $X$ is a metrizable compactum, then so is $\wx^k$, and $F_k(X,e)$, being a continuous Hausdorff image of a metrizable compactum, is also a metrizable compactum.

\begin{lemma}\label{l:conn}
If $X$ is path-connected, then $F(X,e)$ is path-connected.
\end{lemma}
\begin{proof}
The space $\wx$ is path-connected as the union of the two path-connected sets $X$ and $X^{-1}$ with the common point $e$. Hence all the sets $F_k(X,e)=j_k(\wx^k)$ are path-connected. All of them contain $e$, and their union is $F(X,e)$.
\end{proof}

\subsection{The Graev extension of a metric}

\begin{theorem}[{Graev \cite{Graev1948}; see also \cite[Theorem~7.2.2 and Exercise~7.2.a]{ArhangelskiiTkachenko2008}, \cite{Sipacheva2003}}]\label{t:graev}
Let $d$ be a metric on $X$ generating the topology of $X$. There exists an invariant metric $\hd$ on $F(X,e)$ extending $d$ such that $\rho\le \hd$ for every invariant pseudometric $\rho$ on $F(X,e)$ extending $d$. The function $\hd\colon F(X,e)\times F(X,e)\to\Rp$ is continuous \textup(equivalently, the norm $g\mapsto\hd(g,e)$ is continuous\textup).
\end{theorem}

The metric $\hd$ is called the \term{Graev extension} of $d$, and the corresponding invariant norm $N_d(g)=\hd(g,e)$ is called the \term{Graev norm}. An explicit formula for $N_d$ in terms of cancellation schemes of words can be found in \cite{Graev1948,SipachevaUspenskii1987,Sipacheva2003}; it is not used in this paper.

\begin{lemma}\label{l:max}
Let $d$ be a metric on $X$ generating the topology of $X$, and let $\rho$ be an invariant pseudometric on $F(X,e)$ such that $\rho(x,y)\le d(x,y)$ for $x,y\in X$. Then $\rho\le\hd$.
\end{lemma}
\begin{proof}
The function $\rho'=\max(\rho,\hd)$ is an invariant pseudometric on $F(X,e)$, and $\rho'(x,y)=\max(\rho(x,y),d(x,y))=d(x,y)$ for $x,y\in X$, i.e., $\rho'$ extends $d$. By Theorem~\ref{t:graev}, $\rho'\le\hd$, whence $\rho\le\hd$.
\end{proof}

\begin{corollary}\label{c:lip}
Let $d$ be a metric on $X$ generating the topology of $X$, let $K$ be a group with an invariant pseudonorm $\nu$, let $L>0$, and let $f\colon X\to K$ be a map such that $f(e)=e$ and $\nu(f(x)f(y)^{-1})\le L\, d(x,y)$ for all $x,y\in X$. Let $\hat f\colon F(X,e)\to K$ be the homomorphism extending $f$. Then $\nu(\hat f(g))\le L\,N_d(g)$ for all $g\in F(X,e)$.
\end{corollary}
Note that $f$ is not assumed to be continuous.
\begin{proof}
Let $\rho(g,h)=L^{-1}\nu(\hat f(g)\hat f(h)^{-1})$ for $g,h\in F(X,e)$. Since $\hat f$ is a homomorphism and $\nu$ is an invariant pseudonorm, $\rho$ is an invariant pseudometric on $F(X,e)$, and $\rho(x,y)\le d(x,y)$ for $x,y\in X$. By Lemma~\ref{l:max}, $\rho\le\hd$, whence $\nu(\hat f(g))=L\rho(g,e)\le L\,\hd(g,e)=L\,N_d(g)$.
\end{proof}

\section{Construction of the group $G$ and proof of the theorem}
\label{sec:constr}

From now on, $X=[0,1]$, $e=0$, $d(s,t)=|s-t|$, and $N=N_d$ is the Graev norm on $\FI$. The point $1\in[0,1]\subset \FI$ is denoted by $\one$. Throughout, $n\ge665$ is a fixed odd number.

For $c>0$, let
\[
U_c=\{g\in \FI: N(g)<c\},\qquad P_c=\{g^n: g\in U_c\},\qquad H_c=\env{P_c}.
\]
Since the norm $N$ is invariant, the sets $U_c$ and $P_c$ are invariant, and hence the subgroups $H_c$ are normal. Clearly, $H_c\subset H_{c'}$ for $c<c'$. Since $N$ is continuous (Theorem~\ref{t:graev}), the sets $U_c$ are open in $\FI$.

The proof of the theorem is based on the following two propositions, which will be proved in~\S\ref{sec:A} and~\S\ref{sec:B}.

\begin{proposition}\label{p:A}
If $0<c_1<c_2$, then $\cl{H_{c_1}}\subset H_{c_2}$.
\end{proposition}

\begin{proposition}\label{p:B}
$\one^n\notin H_1$.
\end{proposition}

\begin{proof}[Proof of the theorem]
Let $H=\cl{H_{1/2}}$. The closure of a normal subgroup of a topological group is a normal subgroup; thus $H$ is a closed normal subgroup. Let $G=\FI/H$ be the quotient group with the quotient topology, and let $\ph\colon \FI\to G$ be the quotient homomorphism. Since $H$ is closed, $G$ is a Hausdorff topological group; $\ph$ is a continuous open map. The group $G$ is path-connected as a continuous image of the path-connected group $\FI$ (Lemma~\ref{l:conn}).

Let $V=\ph(U_{1/2})$; this is an open neighbourhood of the identity of $G$. If $x\in V$, then $x=\ph(g)$ for some $g\in U_{1/2}$, and $g^n\in P_{1/2}\subset H$, whence $x^n=\ph(g^n)=e$.

Let $x=\ph(\one)$. If $x^n=e$, then $\one^n\in H=\cl{H_{1/2}}\subset H_1$ by Proposition~\ref{p:A}, which contradicts Proposition~\ref{p:B}. Therefore $x^n\neq e$.
\end{proof}

\section{The closure of the subgroups $H_c$}
\label{sec:A}

In this section we prove Proposition~\ref{p:A}.

\subsection{Lipschitz endomorphisms}

We denote by $\wt{[0,1]}=[0,1]\cup[0,1]^{-1}\subset\FI$ the set $\wx$ from~\S\ref{sec:prelim} for $X=[0,1]$, $e=0$; here $0=0^{-1}=e$.

Let $h\colon[0,1]\to[0,1]$ and $h(0)=0$. Since $\FI$ is the free group with basis $(0,1]$, the map $h$ extends uniquely to an endomorphism $\hat h\colon \FI\to\FI$. Let $\wt h(t^{\pm1})=h(t)^{\pm1}$ for $t\in[0,1]$; then $\wt h\colon \wt{[0,1]}\to\wt{[0,1]}$ and
\begin{equation}\label{e:hjk}
\hat h(j_k(x_1,\dots,x_k))=j_k(\wt h(x_1),\dots,\wt h(x_k)).
\end{equation}

\begin{lemma}\label{l:lip}
Let $L>0$, $h\colon[0,1]\to[0,1]$, $h(0)=0$, and $|h(s)-h(t)|\le L|s-t|$ for $s,t\in[0,1]$. Then $N(\hat h(g))\le L\,N(g)$ for $g\in\FI$, and $\hat h(H_c)\subset H_{Lc}$ for $c>0$.
\end{lemma}
\begin{proof}
Apply Corollary~\ref{c:lip} to $K=\FI$, $\nu=N$, and $f=h$: since $\hd$ extends $d$, we have $N(h(s)h(t)^{-1})=|h(s)-h(t)|\le L|s-t|$. Hence $N(\hat h(g))\le L\,N(g)$. Since $N(q^{-1})=N(q)$ and $q^{-n}=(q^{-1})^n$, every element $g\in H_c$ has the form $g=q_1^n\cdots q_k^n$ with $q_i\in U_c$. Then $\hat h(g)=\hat h(q_1)^n\cdots\hat h(q_k)^n$ and $N(\hat h(q_i))\le L\,N(q_i)<Lc$, i.e., $\hat h(g)\in H_{Lc}$.
\end{proof}

\begin{lemma}\label{l:ext}
Let $T\subset[0,1]$ be a finite set with $0\in T$, and let $h^*\colon T\to[0,1]$ be a map such that $h^*(0)=0$ and $|h^*(s)-h^*(t)|\le|s-t|$ for $s,t\in T$. Then $h^*$ extends to a map $h\colon[0,1]\to[0,1]$ such that $|h(s)-h(t)|\le|s-t|$ for $s,t\in[0,1]$.
\end{lemma}
\begin{proof}
Let $T=\{s_0<s_1<\dots<s_p\}$, $s_0=0$. Let $h$ be linear on each segment $[s_{i-1},s_i]$, and let $h(t)=h^*(s_p)$ for $t\ge s_p$. The function $h$ is continuous and piecewise linear, and the absolute values of the slopes of all its pieces do not exceed~$1$; hence $h$ is $1$-Lipschitz.
\end{proof}

\subsection{The metric on $\wt{[0,1]}$}

For $t\in[0,1]$, let $|t|=|t^{-1}|=t$; this defines a function $|\cdot|\colon\wt{[0,1]}\to[0,1]$. Every element $y\in\wt{[0,1]}\setminus\{e\}$ can be uniquely written as $y=t^{\sigma}$ with $t\in(0,1]$ and $\sigma=\pm1$; we call $\sigma$ the \term{sign} of $y$.

\begin{lemma}\label{l:wd}
The restriction $\vd$ of the metric $\hd$ to $\wt{[0,1]}$ is given by
\[
\vd(s^\sigma,t^\sigma)=|s-t|,\qquad \vd(s^\sigma,t^{-\sigma})=s+t\qquad(s,t\in[0,1],\ \sigma=\pm1).
\]
The metric $\vd$ generates the topology of $\wt{[0,1]}$, and $\bigl||y|-|z|\bigr|\le\vd(y,z)$ for $y,z\in\wt{[0,1]}$.
\end{lemma}
\begin{proof}
We have $\hd(s,t)=|s-t|$ and, by invariance, $\hd(s^{-1},t^{-1})=N(s^{-1}t)=N(ts^{-1})=|s-t|$. Further, $\hd(s,t^{-1})=N(st)\le N(s)+N(t)=s+t$. On the other hand, applying Corollary~\ref{c:lip} to $K=\R$, $\nu(v)=|v|$, $f(t)=t$, and $L=1$, we obtain a homomorphism $\hat f\colon\FI\to\R$ with $|\hat f(g)|\le N(g)$, whence $N(st)\ge|s+t|=s+t$. The metric $\vd$ is continuous on the compact space $\wt{[0,1]}$ and therefore generates its topology. The last inequality is obvious from the formulas.
\end{proof}

\subsection{The operation $[\,\cdot\,]_1$ and the closure}

For $M\subset \FI$, let
\[
[M]_1=\bigcup_{k\in\N}\cl{M\cap F_k([0,1],0)}.
\]
Clearly, $M\subset[M]_1\subset\cl M$, and $[M]_1\subset[M']_1$ for $M\subset M'$. By induction on ordinals $\al\le\om_1$ we define $[M]_\al$: $[M]_0=M$, $[M]_{\al+1}=\bigl[[M]_\al\bigr]_1$, and $[M]_\al=\bigcup_{\be<\al}[M]_\be$ for limit $\al$. Since $Q\subset[Q]_1$ for every $Q$, by induction we obtain $[M]_\al\subset[M]_\be$ for $\al\le\be$.

\begin{lemma}\label{l:tf}
$\cl M=[M]_{\om_1}$ for every $M\subset\FI$.
\end{lemma}
\begin{proof}
By induction on $\al$ we obtain $[M]_\al\subset\cl M$; in particular, $[M]_{\om_1}\subset\cl M$. It suffices to check that the set $Q=[M]_{\om_1}$ is closed. Suppose the contrary. By Fact~\ref{f:klim}, there is $k\in\N$ such that the set $P=Q\cap F_k([0,1],0)$ is not closed. Since $F_k([0,1],0)$ is a metrizable compactum, there exist $x\in\cl P\setminus P$ and a sequence $(x_i)_{i\in\N}\subset P$ converging to $x$. For each $i$ we have $x_i\in[M]_{\al_i}$ for some $\al_i<\om_1$. Let $\be=\sup_i\al_i<\om_1$. Then $x_i\in[M]_\be\cap F_k([0,1],0)$ for all $i$, and therefore $x\in\bigl[[M]_\be\bigr]_1=[M]_{\be+1}\subset Q$. Since $x\in F_k([0,1],0)$, we obtain $x\in P$, a contradiction.
\end{proof}

\subsection{The main lemma and the proof of Proposition~\ref{p:A}}

\begin{lemma}\label{l:main}
If $0<c_1<c_2$, then $[H_{c_1}]_1\subset H_{c_2}$.
\end{lemma}
\begin{proof}
Let $k\in\N$ and $y\in\cl{H_{c_1}\cap F_k}$, where $F_k=F_k([0,1],0)$. Let $A=j_k^{-1}(H_{c_1}\cap F_k)\subset\wt{[0,1]}^k$. The map $j_k\colon\wt{[0,1]}^k\to F_k$ is a continuous surjection of a compact space onto a Hausdorff space; hence it is closed and $\cl{j_k(A)}=j_k(\cl A)$. Therefore there exists $\bar y=(y_1,\dots,y_k)\in\cl A$ with $j_k(\bar y)=y$.

Fix $\ga$ with $c_1/c_2<\ga<1$. Let $R=\{0\}\cup\{|y_i|: i=1,\dots,k\}$, $\Delta=\min\{|a-b|: a,b\in R,\ a\neq b\}$ (if $R=\{0\}$, let $\Delta=1$), and $\ep=(1-\ga)\Delta/2$. Since $\bar y\in\cl A$, there exists $\bar z=(z_1,\dots,z_k)\in A$ such that $\vd(y_i,z_i)<\ep$ for all $i$; let $z=j_k(\bar z)\in H_{c_1}$. It is also convenient to put $y_0=z_0=e$.

We check the following properties (for $i,j\in\{0,1,\dots,k\}$).
\begin{enumerate}
\item[(a)] If $|y_i|\ne|y_j|$, then $\ga\bigl||y_i|-|y_j|\bigr|\le\bigl||z_i|-|z_j|\bigr|$. Indeed, by Lemma~\ref{l:wd},
\begin{align*}
\bigl||y_i|-|y_j|\bigr|&\le\vd(y_i,z_i)+\bigl||z_i|-|z_j|\bigr|+\vd(z_j,y_j)<2\ep+\bigl||z_i|-|z_j|\bigr|
\\
&\le(1-\ga)\bigl||y_i|-|y_j|\bigr|+\bigl||z_i|-|z_j|\bigr|,
\end{align*}
since $2\ep=(1-\ga)\Delta\le(1-\ga)\bigl||y_i|-|y_j|\bigr|$.
\item[(b)] If $|z_i|=|z_j|$, then $|y_i|=|y_j|$: this follows from (a). In particular (for $j=0$), if $z_i=e$, then $y_i=e$.
\item[(c)] If $y_i\neq e$, then $z_i\ne e$ and the signs of $y_i$ and $z_i$ coincide. Indeed, $z_i\neq e$ by (b); if the signs were different, then, by Lemma~\ref{l:wd}, $\vd(y_i,z_i)=|y_i|+|z_i|>|y_i|\ge\Delta>\ep$.
\end{enumerate}

Let $T=\{|z_i|: i=0,1,\dots,k\}\ni 0$, and define $h^*\colon T\to[0,1]$ by $h^*(|z_i|)=\ga|y_i|$. By (b), $h^*$ is well defined and $h^*(0)=0$; by (a), $|h^*(s)-h^*(t)|\le|s-t|$ for $s,t\in T$. By Lemma~\ref{l:ext}, extend $h^*$ to a $1$-Lipschitz map $h\colon[0,1]\to[0,1]$; $h(0)=0$. Let also $h_\ga(t)=\ga t$.

We show that $\wt h(z_i)=\wt{h_\ga}(y_i)$ for $i=1,\dots,k$. If $y_i=e$, then $|y_i|=0$ and $\wt h(z_i)=h(|z_i|)^{\pm1}=e=\wt{h_\ga}(y_i)$. If $y_i=t^\sigma\ne e$, then, by (c), $z_i=s^\sigma$ with $s=|z_i|$, and $\wt h(z_i)=h(s)^\sigma=(\ga t)^\sigma=\wt{h_\ga}(y_i)$. By~(\ref{e:hjk}), $\hat h(z)=\widehat{h_\ga}(y)$.

By Lemma~\ref{l:lip} (with $L=1$), $\hat h(z)\in H_{c_1}$, i.e., $\widehat{h_\ga}(y)\in H_{c_1}$. Let $r(t)=\min(t,\ga)/\ga$. Then $r\colon[0,1]\to[0,1]$, $r(0)=0$, $|r(s)-r(t)|\le\ga^{-1}|s-t|$, and $r\circ h_\ga=\mathrm{id}_{[0,1]}$, so that $\hat r\circ\widehat{h_\ga}=\mathrm{id}$. By Lemma~\ref{l:lip} (with $L=\ga^{-1}$), we obtain $y=\hat r(\widehat{h_\ga}(y))\in H_{c_1/\ga}\subset H_{c_2}$, because $c_1/\ga<c_2$.
\end{proof}

\begin{proof}[Proof of Proposition~\ref{p:A}]
By induction on $\al\le\om_1$ we prove the assertion: $[H_c]_\al\subset H_{c'}$ for all $0<c<c'$. For $\al=0$ this is obvious. Let $\al=\be+1$ and $0<c<c'$; put $c''=(c+c')/2$. By the induction hypothesis, $[H_c]_\be\subset H_{c''}$, and by Lemma~\ref{l:main},
\[
[H_c]_{\be+1}=\bigl[[H_c]_\be\bigr]_1\subset[H_{c''}]_1\subset H_{c'}.
\]
For limit $\al$ the assertion follows from the definition of $[H_c]_\al$. For $\al=\om_1$, Lemma~\ref{l:tf} gives $\cl{H_{c_1}}=[H_{c_1}]_{\om_1}\subset H_{c_2}$.
\end{proof}

\section{The element $\one^n$ does not belong to $H_1$}
\label{sec:B}

In this section we prove Proposition~\ref{p:B}.

\subsection{Systems of periods}

Let $\Phi$ be the free group with basis $g_1,\dots,g_m$, and let $\ab\colon\Phi\to\Z^m$ be the abelianization homomorphism, $\ab(g_i)=\mathbf e_i$, where $\mathbf e_1,\dots,\mathbf e_m$ is the standard basis of $\Z^m$. Let
\[
W=\{b\in\Phi:\ \text{at least one coordinate of the vector }\ab(b)\text{ is }0\}.
\]

Let $B(m,n)=\Phi/R$ be the free Burnside group, where $R$ is the normal subgroup generated by the $n$th powers of all elements of $\Phi$.

\begin{definition}\label{d:sp}
Let $m\ge2$. A set of words $\{A_i\}_{i\in I}\subset\Phi$ is called a \term{system of periods} for $B(m,n)$ if the following conditions hold:
\begin{enumerate}
\item[(O1)] every element of $\Phi$ is conjugate modulo $R$ to a power of some $A_i$, i.e., for every $w\in\Phi$ there exist $u\in\Phi$, $i\in I$, $l\in\Z$, and $r\in R$ such that $w=uA_i^lu^{-1}r$;
\item[(O2)] $R/[\Phi,R]$ is a free abelian group with basis $\{A_i^n[\Phi,R]\}_{i\in I}$.
\end{enumerate}
\end{definition}

The algebraic part of the proof reduces to the following lemma and to the existence of systems of periods.

\begin{lemma}\label{l:burnabs}
Let $n$ be odd, $m\ge2$, and suppose that there exists a system of periods for $B(m,n)$. Then $(g_1g_2\cdots g_m)^n\notin\env{\{b^n: b\in W\}}$.
\end{lemma}
\begin{proof}
Let $\{A_i\}_{i\in I}$ be a system of periods for $B(m,n)$. Let $a=g_1\cdots g_m$; then $\ab(a)=(1,\dots,1)$.

By (O1), $a=uA_{i_0}^{l_0}u^{-1}r$ for some $u\in\Phi$, $i_0\in I$, $l_0\in\Z$, $r\in R$. Since $\ab(R)\subset n\Z^m$, we have $(1,\dots,1)\equiv l_0\,\ab(A_{i_0})\pmod n$. Hence $\gcd(l_0,n)=1$: if a prime $p$ divides both $l_0$ and $n$, then $\frac np(1,\dots,1)\equiv\frac{l_0}p\,n\,\ab(A_{i_0})\equiv0\pmod n$, which is false. Let $l_0'\in\Z$ be such that $l_0l_0'\equiv1\pmod n$; then $\ab(A_{i_0})\equiv l_0'(1,\dots,1)\pmod n$.

Let $\bar\Phi=\Phi/[\Phi,R]$ and $\bar R=R/[\Phi,R]$; the subgroup $\bar R$ is central in $\bar\Phi$. The image of an element $w\in\Phi$ in $\bar\Phi$ is denoted by the same letter. By (O2), there exists a homomorphism $\psi\colon\bar R\to\Z/n\Z$ such that $\psi(A_{i_0}^n)=1$ and $\psi(A_i^n)=0$ for $i\ne i_0$.

Let $w\in\Phi$ and $w=uA_i^lu^{-1}r$, where $u\in\Phi$, $r\in R$. Since $r$ and $A_i^n$ lie in the centre of $\bar\Phi$, in $\bar\Phi$ we have
\[
w^n=(uA_i^lu^{-1})^nr^n=uA_i^{ln}u^{-1}r^n=A_i^{ln}r^n\in\bar R,
\]
and therefore $\psi(w^n)=l\,\delta_{i,i_0}\bmod n$ (because $\psi(r^n)=n\psi(r)=0$).

In particular, $\psi(a^n)=l_0\bmod n\neq0$.

Now let $b\in W$ and $b=uA_i^lu^{-1}r$. If $i=i_0$ and $n\nmid l$, then $\ab(b)\equiv l\,\ab(A_{i_0})\equiv l\,l_0'(1,\dots,1)\pmod n$, and all coordinates of $\ab(b)$ are congruent to $l\,l_0'\not\equiv0\pmod n$ (since $n\nmid l$ and $\gcd(l_0',n)=1$); this contradicts the fact that one of the coordinates of $\ab(b)$ is zero. Thus $i\ne i_0$ or $n\mid l$, and in both cases $\psi(b^n)=0$.

The images of all elements $b^n$, $b\in W$, lie in the abelian group $\bar R$ and are annihilated by the homomorphism $\psi$; hence $\psi$ annihilates the image of the subgroup $\env{\{b^n:b\in W\}}$. Since $\psi(a^n)\ne0$, we obtain $a^n\notin\env{\{b^n: b\in W\}}$.
\end{proof}

The existence of systems of periods for large odd $n$ is known in two versions, due to Ol'shanskii and to Adian. Accordingly, Lemma~\ref{l:burnabs} yields two proofs of the algebraic part.

\subsection{First proof: Ol'shanskii's systems of periods}

\begin{fact}[\cite{Olshanskii1982,AshmanovOlshanskii1985,Olshanskii1989book}]\label{f:olsh}
Let $m\ge2$ and let $n>10^{10}$ be odd. Then there exists a system of periods for $B(m,n)$.
\end{fact}
As a system of periods one takes the set $\{A_i\}$ of periods of all ranks used in \cite[Section~18.1]{Olshanskii1989book} to construct a graded presentation of $B(m,n)$ with defining words $A_i^n$ (see also \cite{Olshanskii1982}); by \cite[Theorem~19.7]{Olshanskii1989book}, this presentation defines the free Burnside group, i.e., the normal closure of the words $A_i^n$ is $R$. Condition (O1) is \cite[Theorem~19.4(1)]{Olshanskii1989book}: every element of $B(m,n)$ is conjugate to a power of a period of some rank. Condition (O2) follows from the asphericity of this presentation \cite[Lemma~18.2]{Olshanskii1989book} and \cite[Theorem~31.1(2)]{Olshanskii1989book}: the group $R/[\Phi,R]$ is free abelian with basis $\{A_i^n[\Phi,R]\}$. This is a result of I.\,S.~Ashmanov and A.\,Yu.~Ol'shanskii \cite{AshmanovOlshanskii1985}; it is used in exactly this form in the proof of \cite[Theorem~31.2]{Olshanskii1989book}. In the journal publications these statements are \cite[Lemma~4.6(a)]{Olshanskii1982} (conjugacy to a power of a period; see also the concluding remark of that paper) and \cite[Lemma~3 and Theorem~1]{AshmanovOlshanskii1985} (the basis $\{C_i^n[\Phi,R]\}$ of $R/[\Phi,R]$ for the presentation of $B(m,n)$ from \cite{Olshanskii1982}).

\subsection{Second proof: Adian's elementary periods}

\begin{fact}[\cite{Adian1975,AdianAtabekyan2018}]\label{f:adian}
Let $m\ge2$ and let $n\ge665$ be odd. Then there exists a system of periods for $B(m,n)$.
\end{fact}
As a system of periods one takes Adian's system of elementary periods: the set $\mathcal E=\bigcup_{\al>0}\mathcal E_\al$ of words of the form $A^n$ is defined in \cite[Ch.~VI, 2.1]{Adian1975} (for every elementary word of rank $\al$, exactly one word $A^n$ is chosen in $\mathcal E_\al$), and $\{A_i\}_{i\in I}$ is the set of the corresponding periods $A$. Condition (O1) is \cite[Ch.~VI, 2.5]{Adian1975}: every word $X\ne1$ in $B(m,n)$ is equal to $TE^rT^{-1}$ for some words $T$, $E$ with $E^n\in\mathcal E$ and some integer $r$ (for $X=1$ one can take $r=0$). Condition (O2) was proved by S.\,I.~Adian and V.\,S.~Atabekyan \cite[Theorem~1 and Corollary~2]{AdianAtabekyan2018}: the group $\Phi/[\Phi,R]$ is the free group of rank $m$ of the variety defined by the identity $[x^n,y]=1$, its centre coincides with $R/[\Phi,R]$ and is a free abelian group with basis consisting of the images of the words $A^n\in\mathcal E$.

\subsection{The Burnside lemma}

\begin{lemma}\label{l:burn}
Let $n\ge665$ be odd and $m\ge1$. Then $(g_1g_2\cdots g_m)^n\notin\env{\{b^n: b\in W\}}$.
\end{lemma}
\begin{proof}
For $m=1$ we have $W=\{e\}$, and the assertion is obvious. For $m\ge2$ the assertion follows from Lemma~\ref{l:burnabs} and Fact~\ref{f:adian}; for $n>10^{10}$ one can use Fact~\ref{f:olsh} instead of Fact~\ref{f:adian}.
\end{proof}

\subsection{A weighted norm and the proof of Proposition~\ref{p:B}}

\begin{lemma}\label{l:weight}
Let $0=t_0<t_1<\dots<t_m=1$, $w_k=t_k-t_{k-1}$, and let $\R^m$ be endowed with the norm $\|v\|=\sum_{k=1}^mw_k|v_k|$. Let $f\colon[0,1]\to\R^m$, $f(t)_k=\min\bigl(\max\bigl((t-t_{k-1})/w_k,0\bigr),1\bigr)$, and let $\hat f\colon\FI\to\R^m$ be the homomorphism extending $f$. Then $\|\hat f(g)\|\le N(g)$ for $g\in\FI$, and $f(t_k)=\mathbf e_1+\dots+\mathbf e_k$ for $k=0,1,\dots,m$.
\end{lemma}
\begin{proof}
The equality $f(t_k)=\mathbf e_1+\dots+\mathbf e_k$ is obvious; in particular, $f(0)=0$. Let $s<t$. All coordinates of $f$ are nondecreasing, and $w_k\bigl(f(t)_k-f(s)_k\bigr)$ equals the length of the segment $[s,t]\cap[t_{k-1},t_k]$. Hence $\|f(t)-f(s)\|=t-s$. Since $\|\cdot\|$ is an invariant norm on the abelian group $\R^m$, the assertion follows from Corollary~\ref{c:lip} (with $L=1$).
\end{proof}

\begin{proof}[Proof of Proposition~\ref{p:B}]
Suppose that $\one^n\in H_1$. As in the proof of Lemma~\ref{l:lip}, $\one^n=x_1^n\cdots x_k^n$ with $N(x_j)<1$. Let $T=\{0,1\}\cup\bigcup_{j=1}^k\supp x_j$ and $T=\{t_0<t_1<\dots<t_m\}$; then $t_0=0$, $t_m=1$, $m\ge1$, and all $x_j$ lie in $F(T,0)$, the free group with basis $t_1,\dots,t_m$. The elements $g_i=t_{i-1}^{-1}t_i$, $i=1,\dots,m$ (where $t_0=e$), also form a basis of $F(T,0)$, since $t_i=g_1\cdots g_i$. In particular, $\one=t_m=g_1\cdots g_m$. We identify $F(T,0)$ with the free group $\Phi$ from Lemma~\ref{l:burn} via the basis $g_1,\dots,g_m$.

Let $f$ and $\hat f$ be as in Lemma~\ref{l:weight}. Then $\hat f(g_i)=f(t_i)-f(t_{i-1})=\mathbf e_i$, so that the restriction of $\hat f$ to $F(T,0)$ coincides with the abelianization homomorphism $\ab\colon F(T,0)\to\Z^m$ with respect to the basis $g_1,\dots,g_m$. By Lemma~\ref{l:weight}, $\|\ab(x_j)\|\le N(x_j)<1=\sum_{k=1}^mw_k$. Since the coordinates of the vector $\ab(x_j)$ are integers, and if all of them were nonzero we would have $\|\ab(x_j)\|\ge\sum_kw_k$, at least one coordinate of $\ab(x_j)$ is zero, i.e., $x_j\in W$. Then $(g_1\cdots g_m)^n=\one^n=x_1^n\cdots x_k^n\in\env{\{b^n:b\in W\}}$, which contradicts Lemma~\ref{l:burn}.
\end{proof}

\section{Remarks and questions}
\label{sec:remarks}

\begin{remark}\label{r:665}
The restriction $n\ge665$ comes only from Fact~\ref{f:adian}: the theorem holds for every odd $n$ such that for all $m\ge2$ there exists a system of periods for $B(m,n)$ (Lemma~\ref{l:burnabs}). The preprint \cite{ZyabrevReznichenko2024} used Ol'shanskii's version (Fact~\ref{f:olsh}). The groups $B(m,n)$ are known to be infinite for smaller odd $n$ as well, for instance, for $n\ge101$ \cite{Adian2015}; we do not know whether (O1) and (O2) hold for $101\le n<665$.
\end{remark}

\begin{remark}\label{r:n2}
For $n=2$ the answer to the question is positive. Let $x^2=e$ for all $x$ in a neighbourhood $U$ of the identity of a connected group $G$, and let $W$ be a symmetric neighbourhood of the identity with $W^2\subset U$. For $x,y\in W$ we have $x^2=y^2=(xy)^2=e$, whence $xy=(xy)^{-1}=y^{-1}x^{-1}=yx$. Therefore the connected group $G=\env W$ is abelian, and the set $\{x\in G: x^2=e\}$ is a subgroup of $G$ containing $U$, i.e., an open, and hence also closed, subgroup. Thus it coincides with $G$.
\end{remark}

\begin{question}
Does there exist a connected topological group in which the identity $x^3=e$ holds in a neighbourhood of the identity, but which does not satisfy this identity? The same question for $n=4$ and, more generally, for $2<n<665$, as well as for even $n\ge665$.
\end{question}

\begin{question}
Does the conclusion of Mycielski's theorem hold for the identity $gx^2=x^2g$ with a fixed $g\in G$?
\end{question}

\begin{question}
Can the group $G$ in the theorem be chosen metrizable?
\end{question}


\section*{Funding}
This research did not receive any specific grant from funding agencies in the public, commercial, or not-for-profit sectors.

\section*{Declaration of generative AI and AI-assisted technologies in the manuscript preparation process}
During the preparation of this work the authors used Claude (Anthropic) to translate the manuscript from Russian into English, to help revise the exposition and the proofs, and to check the bibliographic references against the original sources. After using this tool, the authors reviewed and edited the content as needed and take full responsibility for the content of the published article.

\end{document}